\documentclass[10pt,letterpaper]{amsart}

\usepackage{ mathrsfs }
\usepackage{amsxtra}
\usepackage[utf8]{inputenc}
\usepackage{amsthm,amsmath,amssymb,oldgerm,latexsym, amscd, calc, microtype, booktabs}
\usepackage[square]{natbib}
\usepackage{graphicx}
\usepackage[dvipsnames]{xcolor}
\usepackage{caption}
\usepackage{tcolorbox}
\usepackage[all]{xy}

\usepackage{hyperref}
\hypersetup{
    colorlinks,
    citecolor=blue,
    filecolor=black,
    linkcolor=red,
    urlcolor=black
}
\bibpunct{[}{]}{,}{n}{}{;}

\numberwithin{equation}{section}
\newcommand\Hom{{\mathrm {Hom}}}

\newcommand\C{\mathbb{C}}

\newcommand{\tH}{\mathbb{H}}
\newcommand{\Z}{\mathbb{Z}}
\newcommand{\T}{\mathbb{T}}

\newcommand{\N}{\mathbb{N}}
\newcommand{\F}{\mathbb{F}}

\newcommand{\Ga}{\Gamma}

\renewcommand{\Re}{\operatorname{Re}}

\newcommand{\HH}{\mathrm{H}}

\begin{document}
\theoremstyle{plain}
\newtheorem{thm}{Theorem}[section]
\newtheorem{lem}[thm]{Lemma}
\newtheorem{prop}[thm]{Proposition}
\newtheorem{cor}[thm]{Corollary}
\theoremstyle{definition}
\newtheorem{defn}[thm]{Definition} 
\theoremstyle{definition}
\newtheorem{rem}[thm]{Remark}
\newtheorem{notn}[thm]{Notation}
\newtheorem{notns}[thm]{Notations}
\newtheorem{example}[thm]{Example}
\newtheorem{Ques}[thm]{Question}

\makeatletter
\def\imod#1{\allowbreak\mkern10mu({\operator@font mod}\,\,#1)}
\makeatother
\title[Linear independence of Fourier coefficients of modular forms]{Modular forms with non-vanishing central values and linear independence of Fourier coefficients}

\author{Debargha Banerjee}
\author{Priyanka Majumder}

\address{Debargha Banerjee, Department of Mathematics, Indian Institute of Science Education and Research (IISER) Pune 411008, India}
\email{debargha.banerjee@gmail.com}

 \address{Priyanka Majumder, Theoretical Statistics and Mathematics Unit, Indian Statistical Institute, Bangalore Centre, 8th Mile, Mysore Road, 560059 Bangalore, India}
\email{pmpriyanka57@gmail.com }

\thanks{The first named author is partially supported by the SERB grant MTR/2017/000357 and CRG/2020/000223. 
        The second named author was partially supported by the  CRG/2020/000223 and IISER Pune post-doctoral fellowship. 
        The article came out of the discussion of the first author with Professor Peter Sarnak. 
        The second named author would like to thank Dr. Pramath Anamby for helpful discussions. It is a pleasure to acknowledge several fruitful  email communication with Professors  Lo\"ic Merel, Jefferey VanderKam and Satadal Ganguly. We are sincerely grateful to the anonymous referee for carefully reading our paper that helped us to improve mathematical content and exposition. }
\begin{abstract} 
In this article, we are interested in  modular forms with non-vanishing central critical values and linear independence of Fourier coefficients of modular forms. 
The main ingredient is a generalization of a theorem due to VanderKam  to modular symbols of higher weights.
We prove that for sufficiently large primes $p$, Hecke operators $T_1, T_2, \ldots, T_D$ act linearly independently on the winding elements inside the space of weight $2k$  cuspidal modular symbol $\mathbb{S}_{2k}(\Gamma_0(p))$ with $k\geq 1$ for $D^2\ll p$. This gives a bound on the number of newforms with non-vanishing arithmetic $L$-functions at their central critical points and linear independence on the reductions of these modular forms for  prime modulo $l\not=p$.
\end{abstract}
\maketitle

\section{Introduction}
There is a considerable interest to study Hecke algebras generated by Hecke operators acting on the space of classical elliptic modular forms or dually the space of modular symbols. Hecke operators act by correspondences on the modular curves, and hence on the modular forms and modular symbols. These are important for the study of integral or rational structures.  These algebras play a crucial rule in some important discoveries of number theory, including Merel's proof of uniform boundedness of torsion of elliptic curves. Studying linear independence of Hecke operators acting on arithmetic cycles play an important role in Merel's proof of uniform boundedness conjecture~\cite{MR1369424}. One of the key ingredient of loc. cit 
is a Proposition (cf ~\cite[Proposition 3]{MR1369424} that guarantees that given an integer $d>3$ there exists a bound $p$ 
such that if $p > B(d)$, then Hecke the Hecke operators $T_1, T_2, \ldots, T_D$ act linearly independently on the winding element $\mathbf{e}$, which is the image of the algebraic cycle $\{0, \infty\}$ inside the space of cuspidal modular symbols. That basically ensures there is $p$ torsion of any elliptic curve over number field with  extension degree $d>3$.

In \cite{MR1760688}, VanderKam improved the bound using analytic techniques. In particular, VanderKam proved that the Hecke operators $T_1, T_2, \ldots, T_D$ act linearly independent on the cycle $\mathbf{e}$ when $D^2\ll p$. 
It is natural to ask given a fixed $D>3$, how big should be $p$ to ensure this linear dependence and what the implications will be for modular forms of higher weights.

In this article, we extend the result to modular forms of arbitrary weight $k >2$. The linear independence of Hecke operators acting on algebraic cycles have several implications towards bound on the basis of modular forms  with   simultaneous non-vanishing of $L$-functions  similar to \cite[p. 286]{MR2019017} and \cite{MR1934306}. 

We are interested in the congruence subgroups of the form $\Gamma_0(p)$ for a prime $p$ that parametrizes a cyclic subgroup of order $p$ for a prime $p \geq 5$. 
Let $\mathcal{H}= \lbrace \tau \in \mathbb{C}\mid \mathrm{Im}(\tau)>0 \rbrace$ be the upper half plane and consider compactification $\mathcal{H}^*= \mathcal{H}\cup \mathbb{P}^1{(\mathbb{Q})}$. Let us consider the modular curve $X_0(p):= \Gamma_0(p)\backslash \mathcal{H}^*= Y_0(p)\cup \partial(X_0(p))$ obtained by adding the cusps. Consider the space of cuspidal modular symbols, that is the homology group $\HH_1(X_0(p),\Z)$. We denote by $\mathbb{S}_2(\Gamma_0(p))$ the space of weight $2$ the cuspidal modular symbols (see \cite[Theorem 1.9]{MR0345909}).
The Hecke algebra $\mathbb{T}_\Z$ acts on homology $\HH_1(X_0(p), \mathbb{Z})\cong \mathbb{S}_2(\Gamma_0(p))$. VanderKam \cite{MR1760688} proved the linear independence of Hecke operators $T_1, T_2, \ldots, T_D$ acting on the winding element $\mathbf{e}$ when $p> c_\delta \, D^{2+ \delta}$ for any given $\delta>0$ and $c_\delta$ is an effective constant. Note that the image of $\lbrace 0,  \infty\rbrace \in \mathbb{S}_2(\Gamma_0(p))$ 
is the winding element $\mathbf{e}$.

In our present article we prove the linear independence of the Hecke operators on the algebraic cycle of $\mathbb{S}_{2k}(\Gamma_0(p))$ for all $k \geq 1$.
Fix a prime $p \geq 5$.  Let $S_{2k}(\Gamma_0(p))$ be the space of cusp forms of weight $2k$ for the congruence subgroup $\Ga_0(p)$.
For a Hecke eigenform $f=\sum\limits_{n\geq 1} a_n(f)  q^n  \in S_{2k}(\Gamma_0(p))$, let $L^{ar}(f,s)$ be the $L$-function with arithmetic normalization as defined in \S~\ref{Lfunction}. We are interested in newforms with non-vanishing arithmetic $L$-functions at the central critical values (cf. \S~\ref{Lfunction}).
Let $S_{2k}(\Gamma_0(p), \F_l)$ be the mod $l$ modular form in the sense 
of Serre for primes $l \neq p$.

For a modular form $f \in S_{2k}(\Gamma_0(p))$ with Fourier coefficients belonging to $\Z$, let 
$\bar{f}$ be the modular form in $S_{2k}(\Gamma_0(p), \F_l)$ obtained
by reducing all coefficients in $\infty$ modulo $l$. 
We establish a bound on the number of such Hecke eigen forms in terms of $p$.  In other words, we prove the following theorem:

\begin{thm}\label{mainthm}
Let $k\geq 1$ be a fixed positive integer and $p$ be a prime number $p \geq 5$. For a given $\delta>0$ there exists an explicit  constant $c_\delta$ and consider $D \in \N$ such that  $p> c_\delta D^{2+\delta}$. We have $D$ many Hecke eigenforms $f_1, f_2, \ldots, f_D \in S_{2k}(\Gamma_0(p))$ that satisfy the following conditions:
\begin{itemize}
\item[(a)] 
 The central critical values $L^{ar}(f_i,k) \neq 0$ with $i=1, \ldots, D$.
\item[(b)]
Consider the modular forms 
$\bar{f}_1, \bar{f}_2, \ldots, \bar{f}_D \in S_{2k}(\Gamma_0(p), \F_l)$ for $l \neq p$.  
They  lie in the annhilator of the winding element and the vectors 
$
(a_1(\bar{f}_i), a_2(\bar{f}_i), \ldots, a_D(\bar{f}_i))_{i \in \{1,\ldots,D\}}
$
are linearly independent  over $\F_l$. 
\end{itemize}
 \end{thm}
The main ingredient to prove Theorem \ref{mainthm} is a generalization of a theorem due to VanderKam \cite{MR1760688} that proves the linear independence of Hecke operators $T_1, T_2, \ldots, T_D$ acting on the winding element $\mathbf{e}$ (cf.\S~\ref{winding}) when $p> c_\delta \, D^{2+ \delta}$ for any given $\delta>0$ and $c_\delta$ an effective constant. 

Let $\mathbb{S}_{2k}(\Gamma_0(p))$ be the space of cuspidal modular symbols of weight $2k$ with $k\geq1$ as defined in \S~\ref{modsym}. This is the homology group of the modular curve $X_0(p)$  with coefficient in a locally constant sheaf (rather than the constant sheaf). Let $\mathbf{e}$ be the winding element belonging to $\mathbb{S}_{2k}(\Gamma_0(p))$,
In our present article, we prove the linear independence of the Hecke operators on the algebraic cycle of the form $z^{k-1} \otimes \mathbf{e}$ inside $\mathbb{S}_{2k}(\Gamma_0(p))$ for all $k \geq 1$.
This winding element is important because it gives us information of central values of the arithmetic $L$-functions of  newforms of arbitrary weight.

Linear independence of $L$ values is an active area of mathematics pursued in~\cite{MR4500744} and \cite{MR382183}. In some sense, this is mod $l$ version of Strum bound. However, what is curious is that Strum bound depends crucially on weights of space of modular forms. However, our bound is independent of weights. This is because in this paper  we are only interested in modular forms with non-vanishing central critical values. We end the computation by some numerical computation using softwares {\tt MAGMA } and {\tt SAGE} that gives a way to detect linear independence of Hecke operators on winding elements.

\section{Preliminaries}
\subsection{Modular symbols}
\label{modsym}
We are  interested in the congruence subgroups of the form $\Gamma_0(p)$  that parametrizes cyclic subgroups of order $p$ for a prime $p \geq 5$. 
Let $\mathcal{H}= \lbrace \tau \in \mathbb{C}\mid \mathrm{Im}(\tau)>0 \rbrace$ be the upper-half plane and consider the compactification $\mathcal{H}^*= \mathcal{H}\cup \mathbb{P}^1{(\mathbb{Q})}$. Let us consider the modular curve $X_0(p):= \Gamma_0(p)\backslash \mathcal{H}^*= Y_0(p)\cup \partial(X_0(p))$ obtained by adding the cusps. 
We define the modular symbols $\mathbb{M}_{2k}(\Ga_0(p))$ associated to the congruence subgroup $\Ga_0(p)$ of weight $2k $ with $k\geq 1$  following \cite{MR2289048}.

Let $\mathbb{M}_2$ be the free abelian group with basis the set of symbols $\lbrace \alpha, \beta\rbrace$ with $\alpha, \beta \in \mathbb{P}^1(\mathbb{Q})$ modulo the $3$-term relations
\begin{align*}
    \lbrace \alpha, \beta\rbrace+ \lbrace  \beta, \gamma\rbrace+\lbrace \gamma, \alpha\rbrace=0 
\end{align*}
for all $\alpha, \beta, \gamma \in \mathbb{P}^1(\mathbb{Q})$, and for all torsion, i.e.,
\begin{align*}
    \mathbb{M}_2=(F\slash R)\slash(F\slash R)_{\mathrm{tor}},
\end{align*}
where
\begin{itemize}
\item 
$F$  is  the free abelian group consisting of all pairs $\{\alpha, \beta\}$,
\item 
$R$ is the subgroup generated by all elements of the form  $ \{ \alpha, \beta\}+ \{  \beta, \gamma \}+\{ \gamma, \alpha\}$.
\end{itemize}
 The group $\mathbb{M}_2$ is the group of modular symbols of weight $2$.  The left action of $\Gamma_0(p)$ on $\mathbb{M}_2$ is defined as follows:
\begin{align*}
    g\lbrace\alpha, \beta\rbrace=\lbrace g(\alpha), g(\beta)\rbrace,
\end{align*}
where $g\in \Gamma_0(p)$ acts via the fractional linear transformation
\begin{align*}
    g(\alpha)= \frac{a\alpha+b}{c\alpha+d}, \,\ \text{where}\,\ g=\left(\begin{array}{ccc} a &b\\ 
 c & d \end{array}\right).
\end{align*}

For any integer $n\geq 0$, let $\mathbb{Z}[X,Y]_n$ be the abelian group of homogeneous polynomials of degree $n$ in two variables $X, Y$. Recall that this defines a locally constant sheaf $\mathcal{F}_{n}$ on the modular curve $X_0(p)$. 
Note that $\mathbb{Z}[X,Y]_n$ is isomorphic to $\mathrm{Sym}^{n}(\mathbb{Z}\times\mathbb{Z}))$ as a group. For a fixed integer $k \geq 1$, we define
\begin{align*}
    \mathbb{M}_{2k}:=\mathbb{Z}[X,Y]_{2k-2}\otimes_{\mathbb{Z}} \mathbb{M}_2,
\end{align*}
which is a torsion-free abelian group whose elements are sums of expressions of the form $X^iY^{2k-2-i}\otimes \lbrace\alpha, \beta \rbrace$. The left action of $\Gamma_0(p)$ on $\mathbb{Z}[X,Y]_{2k-2}$ is defined as follows: Let $g=\left(\begin{array}{ccc} a &b\\ 
 c & d \end{array}\right)\in \Gamma_0(p)$ and $P(X,Y)\in \mathbb{Z}[X,Y]_{2k-2}$, then we have
 \begin{align*}
     (gP)(X,Y)=P(dX-bY,-cX+aY)
 \end{align*}
 and the left action of $\Gamma_0(p)$ on $\mathbb{M}_{2k}$ is given by
 \begin{align*}
     g(P\otimes \lbrace\alpha,\beta\rbrace)= (gP)\otimes\lbrace g(\alpha), g(\beta)\rbrace.
 \end{align*}
\begin{defn}
 The space $\mathbb{M}_{2k}(\Gamma_0(p))$ of weight $2k$ modular symbols for $\Gamma_0(p)$ is the quotient of $\mathbb{M}_{2k}$ by all relations $gx-x $ for $x\in \mathbb{M}_{2k}$, $g\in \Gamma_0(p)$, and by any torsion.
\end{defn}
Let $P \in\mathbb{Z}[X,Y]_{2k-2}$ and $g\in \Gamma_0(p)$, then the \textit{Manin symbol} associated to this pair is given by
\begin{align*}
    [P,g]= g(P\otimes\lbrace 0, \infty\rbrace).
\end{align*}
Note that the Manin symbols generate the space of modular symbols $\mathbb{M}_{2k}(\Gamma_0(p))$ (cf. \cite[Proposition 8.3]{MR2289048}).

\subsection{Hecke operators acting on modular symbols}Here we describe the action of the Hecke operators on the space of modular symbols.
For a prime $q$, let
\begin{align*}
    R_q =\left\lbrace\left(\begin{array}{ccc} 1 &r\\ 
 0 & q \end{array}\right)\mid r=0,1,\ldots, q-1\right\rbrace \cup \left\lbrace\left(\begin{array}{ccc} q &0\\ 
 0 & 1  \end{array}\right)\right\rbrace.
\end{align*}
Then the action of the Hecke operator $T_q$ on $\mathbb{M}_{2k}(\Gamma_0(p))$ is given by
\begin{align*}
    T_q(P\otimes \lbrace\alpha,\beta\rbrace)= \sum_{g\in R_p}g(P\otimes \lbrace\alpha,\beta\rbrace).
\end{align*}
Let $\mathbb{B}$ be the free abelian group on the symbols $\lbrace\alpha\rbrace$ with $\alpha\in \mathbb{P}^1(\mathbb{Q})$, and we set
\begin{align*}
    \mathbb{B}_{2k}:= \mathbb{Z}[X,Y]_{2k-2}\otimes \mathbb{B}.
\end{align*}
The left action of $\Gamma_0(p)$ on  $\mathbb{B}_{2k}$ is given by
\begin{align*}
    g (P \otimes\lbrace \alpha\rbrace)=(gP)\otimes\lbrace g(\alpha)\rbrace\,\ \text{for}\,\  P\otimes\lbrace \alpha\rbrace\in \mathbb{B}_{2k},\, g\in \Gamma_0(p).
\end{align*}
Let $k\geq 1$ be an integer and let $\mathbb{B}_{2k}(\Gamma_0(p))$ be the quotient of $\mathbb{B}_{2k}$ by the relations $x-gx$ for all $g\in \Gamma_0(p)$, $x\in \mathbb{B}_{2k}$, and by any torsion. Thus, $\mathbb{B}_{2k}(\Gamma_0(p))$ is a torsion-free abelian group.
\begin{defn}
     The space $\mathbb{S}_{2k}(\Gamma_0(p))$ of cuspidal modular symbols of weight $2k$ is the kernel of boundary map
    $
        \delta_{2k}: \mathbb{M}_{2k}(\Gamma_0(p))\to \mathbb{B}_{2k}(\Gamma_0(p))
    $
    which is given by extending the map $\delta_{2k}(P\otimes\lbrace\alpha, \beta\rbrace)= P\otimes\lbrace\beta\rbrace-P\otimes\lbrace\alpha\rbrace$ linearly.
\end{defn}
Let $S_{2k}(\Gamma_0(p))$ be the space of cusp forms equipped with the Petersson inner product:
\begin{align*}
    \langle f, g \rangle_{\mathrm{pet}}= \int_{\Gamma\backslash\mathbb{H}} f(z) \overline{g(z)}\, (\mathrm{Im}(z))^{2k}\, \mu_{\mathrm{hyp}}(z)\,\ \text{for}\,\ f, g\in S_{2k}(\Gamma_0(p)).
\end{align*}

    The dimensions of the space of the cuspidal modular symbols and the space of cusp forms are related as below (cf. \cite[Corollary 8.19]{MR2289048}) (see also \cite[Theorem 0.2]{MR571104} and \cite[\S 1.5]{MR1322319})
    \begin{align*}
        \mathrm{dim}_{\mathbb{C}}\, \mathbb{S}_{2k}(\Gamma_0(p), \mathbb{C})= 2\, \mathrm{dim}_{\mathbb{C}}\, {S}_{2k}(\Gamma_0(p))\quad \text{for} \quad k\geq1.
    \end{align*}

   Note that we may identify the space of modular symbols with certain homology group with coefficient in a locally constant sheaf  following \cite[Theorem 4]{MR0427234}. Let $\mathcal{F}_{2k-2}$ be the usual locally constant sheaf on $X_0(p)$ with a prime $p$. We have the following identification 
    \[
    \mathbb{S}_{2k}(\Gamma_0(p)) \cong H_1 \left(X_0(p), \, \mathcal{F}_{2k-2}\right).
    \]
  
   Note that the number of Hecke algebra generators on $\Z$ that act on the space of modular forms is bounded by a linear function of the prime $p$ (see \cite[Theorem 9.23]{MR2289048}). However, this information is not sufficient to determine the set of basis with non-vanishing $L$-functions.

\subsection{Modular symbols and modular forms}
\label{Lfunction}
For any cusp form $f\in S_{2k}(\Gamma_0(p))$, the integration pairing (see \cite{MR2289048}, p. 180) is given by
\begin{align*}
   \big \langle z^m\lbrace 0, \infty\rbrace, f \big\rangle=\int_0^{ \infty} f(z) z^m dz\quad \text{for}\quad 0\leq m \leq 2k-2. 
\end{align*}
Let the Fourier expansion of a newform $f\in S_{2k}(\Gamma_0(p))$ be of the form $\displaystyle f=\sum_{n\geq 1}a_f(n) q^n$.  Then the $L^{ar}$-function attached to $f$ is defined by
 \begin{align*}
     L^{ar}(f, s):=\sum_{n\geq 1}\frac{a_f(n)}{n^s},
 \end{align*}
 which is absolutely convergent for $\Re(s)>\frac{2k+1}{2}.$
 The $L^{ar}$-function has an Euler product expansion of the form
 \begin{align*}
 L^{ar}(f, s)= \left(1- a_f(p)\, p^{-s}\right)^{-1} \prod_{\substack{q\not= p\\ q \, \mathrm{prime}}} \left(1- a_f(q)\, q^{-s}+q^{2k-1}q^{-2s}\right)^{-1}.
 \end{align*}
 The completed $L^{ar}$-function is defined by
 \begin{align}\label{Lambda}
     \Lambda(f, s):= \left(\frac{\sqrt{p}}{2\pi} \right)^s\Gamma(s)\, L^{ar}(f, s)
 \end{align}
 which satisfies the functional equation $\Lambda(f, s)= i^{-2k}\, \epsilon_f\,  \Lambda(f, 2k- s)$, where $\epsilon_f=\pm 1$.
 
Also, for a newform $f\in S_{2k}(\Gamma_0(p))$ we have the following relation (see \cite{MR3941495}, equation(14), p. 1385)
\begin{align}\label{lfun}
   \big \langle z^m\lbrace 0, \infty\rbrace, f \big\rangle = \frac{m!\, i^{m+1}}{(2\pi)^{m+1}}\, L^{ar}(f, m+1)\,\ \text{for}\,\ 0\leq m \leq 2k-2,
\end{align}
where the $L^{ar}$-function has critical values at the integers $m+1$.
\subsection{Different normalization of \textit{L}-function} 
Two different normalizations of $L$-function exist in the literature. Let $f(z)=\sum\limits_{n \geq 1} \lambda_f(n) n^{\frac{2k-1}{2}} e^{2\pi i n z}$ be a Hecke eigenform in $S_{2k}(\Gamma_0(p))$. We define:
 \[
 L^{an}(f,s):=\sum_{n \geq 1} \frac{\lambda_f(n)}{n^s}
\]
which is absolutely convergent for $\Re(s)>1$. By our definition we also have
 \[
 L^{ar}(f,s)=\sum_{n \geq 1}  \frac{\lambda_f(n)n^{\frac{2k-1}{2}}}{n^s}.
 \]
  Then we get $L^{ar}(f,s)=L^{an}(f,s-(2k-1)/2)$. Since it is well known that the critical strip for $L^{an}(f,s)$ is $0\leq \Re(s) \leq 1$, we get $\frac{2k-1}{2}\leq \Re(s)\leq \frac{2k+1}{2}$ as the critical strip for $L^{ar}(f,s)$.
  The completed analytic $L$-function is defined by
  $$\Lambda^{an} (f,s):=  p^{\frac{s}{2}}L_{\infty}(f,s) L^{an}(f,s),$$
  where $L_{\infty}(f,s)$ is given by
 \[
 L_{\infty}(f,s)= \pi^{-s} \, \Ga \left(\frac{s}{2} +\frac{2k-1}{4}\right)  \Ga \left(\frac{s}{2} +\frac{2k+1}{4}\right).
 \]
 The completed analytic $L$-function satisfies the function equation 
 \[
 \Lambda^{an} (f,s)= \epsilon_f \Lambda^{an} (f,1-s),
 \]
 where $\epsilon_f=\pm 1$. 
    
\subsection{Winding elements}
\label{winding}
 We denote by $\mathbb{S}_2(\Gamma_0(p))$ the space of weight $2$ cuspidal modular symbols (see \cite[Theorem 1.9]{MR0345909}).
The Hecke algebra $\mathbb{T}:=\mathbb{T}_\Z$  is a $\Z$ algebra consisting of linear endomorphisms that act on the $S_{2k}(\Gamma_0(p))$ and on dual space of  modular symbols $\HH_1(X_0(p), \mathbb{Z})\cong \mathbb{S}_2(\Gamma_0(p))$. 
The winding elements are the elements of the space of modular symbols whose annihilators define ideals
of the Hecke algebras with the $L^{ar}$-functions of the corresponding quotients of the Jacobian non-zero.

\begin{defn}
    Let $\{0,\infty\}$ denote the projection
of the path from $0$ to $\infty$ in $\tH \cup {\mathbb{P}}^1(\mathbb{Q})$ to $X_0(p)(\C)$. 
We have an isomorphism $\mathrm{H}_1(X_0(p),\Z)\otimes \mathbb{R} \simeq \mathrm{Hom}_{\C}(\mathrm{H}^0(X_0(p),\Omega^1),\mathbb{C})$.
Let $ \mathbf{e} \in\mathrm{H}_1(X_0(p),\Z)\otimes \mathbb{R} $ correspond to the homomorphism
$\omega \rightarrow \int_0^{\infty} \omega$.
 The element $\mathbf{e}$ is called the {\it winding element}. 
\end{defn}
We wish to emphasize that $\mathbf{e}$ depends on the level $p$ that we are hiding from notation for brevity. Recall that the explicit expression of $\mathbf{e}$ is the
key tool in the Merel's proof of the uniform boundedness theorem~(see \cite{MR1369424}). One key step in Merel's proof is to show the linear independence of the Hecke operators $T_1, T_2, \ldots, T_D$ on the image $\mathbf{e}$  for $k=1$ if $p > D^{3D^2}$ \cite[Proposition 3, p. 440]{MR1369424}. In \cite{MR1760688}, VanderKam improved the Merel's bound for $D$ by using analytic techniques. In particular, VanderKam proved that Hecke operators $T_1, T_2, \ldots, T_D$ act linearly independent on the cycle $\mathbf{e}$ when $p> c_\delta \, D^{2+ \delta}$.

In this article, we extend the linear independence result of VanderKam to the modular forms of arbitrary weight $2k$ with $k\geq 1$. The linear independence of Hecke operators acting on algebraic cycles have several implications towards bound on the basis of modular forms  with simultaneous non-vanishing of $L$-functions inside the critical strip (e.g., see \cite[p. 286]{MR2019017} and \cite{MR1934306}). 

\section{Linear dependence of Hecke operators}
In this section we are interested in the linear dependence of Hecke operators inside the space of cuspidal symbols $\mathbb{S}_{2k}(\Gamma_0(p))$ with a prime $p$.  Suppose that the Hecke operators $T_1, T_2, \ldots, T_D$ are linearly dependent on $z^{k-1} \otimes \mathbf{e}$. 
 Since the winding element $\mathbf{e}$ is the image of $\lbrace 0, \infty\rbrace$,  there exist $\alpha_1, \ldots, \alpha_D$ (at least one of them non-zero) with
\begin{align*}
    \sum_{i=1}^D
 \alpha _i\, T_i( z^{k-1}\otimes\lbrace 0, \infty\rbrace)=0.
 \end{align*} 
Let $\mathcal{B}_{2k}(p)$ denote the Hecke basis of $S_{2k}(\Gamma_0(p))$. In other words, $\mathcal{B}_{2k}(p)$ consists of Hecke eigenforms. 
 Let $f$ be a Hecke eigenform, then by the self-adjoint property (see \cite[Theorem 8.21]{MR2289048}) of the Hecke operator, we have 
 \begin{align*}
     \langle T_i ( z^{k-1}\otimes\lbrace 0, \infty\rbrace), f \rangle=\langle  (z^{k-1}\otimes\lbrace 0, \infty\rbrace), \, T_if \rangle\quad \text{for}\,\ i= 1,\ldots, D.
 \end{align*}
 Let $\lambda_f(i)$ be the Hecke eigenvalue, i.e., $T_if=\lambda_f(i)f$. Then we have
  \begin{align*}
     \langle T_i (z^{k-1}\otimes\lbrace 0, \infty\rbrace), f \rangle=\lambda_f(i)\, \langle  (z^{k-1}\otimes\lbrace 0, \infty\rbrace), \, f \rangle\quad \text{for}\,\ i= 1,\ldots, D.
 \end{align*}
 Now, the linear dependence of the Hecke operators is equivalent to the existence of non-trivial solutions $\alpha_1, \ldots, \alpha_{D}$ to the following system of equations:
 \begin{align}\label{modsquare}
     0= \sum_{f \in \mathcal{B}_{2k}(p)} \omega_f \, \bigg| \sum_{i=1}^D
 \alpha _i\, \langle T_i(z^{k-1}\otimes\lbrace 0, \infty\rbrace), f \rangle \bigg|^2,
 \end{align}
 where $\omega_f=\frac{1}{4\pi} \langle f, f \rangle_{\mathrm{pet}}$. The equation \eqref{modsquare} is equivalent to
 \begin{align*}
     0=\sum_{i=1}^D\sum_{j=1}^D \alpha _i \bar{\alpha} _j \sum_{f\in \mathcal{B}_{2k}(p)} \omega_f \lambda_f(i) \lambda_f(j)\big|\langle (z^{k-1}\otimes\lbrace 0, \infty\rbrace), f \rangle \big|^2.
 \end{align*}
Then, using \eqref{lfun}, we get 
 \begin{align}\label{modsquareLfunction}
     0=\sum_{i=1}^D\sum_{j=1}^D \alpha _i \bar{\alpha} _j \sum_{f\in \mathcal{B}_{2k}(p)} \omega_f \lambda_f(i) \lambda_f(j)\big|L^{ar}(f, k) \big|^2.
 \end{align}

Recall the following Petersson formula (cf. \cite[Lemma 1.2]{MR1760688}) that works for arbitrary weight $k\geq 1$ (see also  \cite[\S 3, p. 1022]{MR3060872}).  In the following lemma, a basis of $S_{2k}(\Gamma_0(p))$
is a Hecke basis if it is an eigenform of Hecke operators $T_n$ for all $n$ with $(n,p)=1$. 
Recall that Hecke eigenvalues satisfy the following recurssion formula:

  \begin{lem}[Petersson formula]\label{petersson}
        If $\mathcal{B}_{2k}(p)$ is a Hecke basis of $S_{2k}(\Gamma_0(p))$ then we have
        \begin{align*}
             \sum_{f\in \mathcal{B}_{2k}(p)}\omega_f  \lambda_f(r) \lambda_f(s)= \delta_{rs}+2 \pi i^{2k} \sum_{p|c}\frac{\mathcal{S}(r, s, c )}{c} \,J_{2k-1}\left(\frac{4\pi \sqrt{rs}}{c}\right),
        \end{align*}
        where $\mathcal{S}(r, s, c)$ denotes the Kloosterman sum and $J_\nu(x)$ denotes the $J$-Bessel function of order $\nu$.
    \end{lem}
   Note that, we use the Petersson formula (i.e., Lemma \ref{petersson}) many times in this paper. For notational convenience, whenever we use the Petersson formula by $S_{\mathrm{main}}$ we denote the term coming from the $\delta_{rs}$-part of the Petersson formula and $S_{\mathrm{off}}$ denote the remaining terms coming from the Kloosterman sum in the Petersson formula.

 \begin{lem}\label{Lfunctionatk}
 Let $f\in  S_{2k}(\Gamma_0(p))$ be a cusp form of weight $2k$ and let $L^{ar}(f,s)$ be the $L^{ar}$-function associated with the cusp form $f$. Then at the central critical point $s = k$, the $L^{ar}$-function has the following expression:
 \begin{align*}
     \big|L^{ar}(f, k) \big|^2=\frac{2}{\left((k-1) !\right)^2} \sum_{l,\, m \geq 1} G_{k}\left(\frac{lm}{p}\right)\frac{\lambda_f(l) \lambda_f(m)}{\sqrt{lm}},
 \end{align*}
 where $G_{k}(x)$ has the following integral representation 
 \begin{align*}
     G_{k}(x)=\frac{1}{2\pi i}\int_{\mathrm{Re}(t)=3\slash 4} \frac{\Gamma(\alpha+t)^2}{(2\pi)^{2t}x^t}\frac{dt}{t}.
 \end{align*}
 \end{lem}
 \begin{proof}
Note that, although the proof of this lemma is closely resembles VanderKam \cite[Lemma 1.1]{MR1760688}, we include the complete proof here in order to provide a comprehensive explanation in our paper.

For values of $s$ outside the critical strip, the completed $L^{ar}$-function satisfies the functional equation $$
     \Lambda(f,s)=i^{-2k} \epsilon_f \Lambda(f, 2k-s),
$$
 where we have the expression for $\Lambda(f, s)$ as given in equation \eqref{Lambda}, and $\epsilon_f$ represents a constant that can take values $\pm 1$. 
 For $s=k+t$ with $t>\frac{1}{2}$ the $L$-function $L^{ar}(f, k+t)$ satisfies this functional equation, i.e., we have
 \begin{align*}
     &\left(\frac{\sqrt{p}}{2\pi}\right)^{k+t}\Gamma(k+t)\, L^{ar}(f, k+t)=i ^{-2k}\epsilon_f\left(\frac{\sqrt{p}}{2\pi}\right)^{k-t}\Gamma(k-t)\, L^{ar}(f, k-t)\\
 \implies&
     \left(\frac{p}{4 \pi^2}\right)^t \Gamma(k+t)^2 L^{ar}(f, k+t)^2= \left(\frac{p}{4 \pi^2}\right)^{-t} \Gamma(k-t)^2 L^{ar}(f, k-t)^.
 \end{align*}
 This implies $\left(\frac{p}{4 \pi^2}\right)^t \Gamma(k+t)^2 L^{ar}(f, k+t)^2$ is an even function. Hence, we can integrate this function against $\frac{dt}{t}$ at $\mathrm{Re}(t)=\frac{3}{4}$ and shift to $\mathrm{Re}(t)=-\frac{3}{4}$ by picking up the residue at $t=0$.
 From the defintion we have
 \begin{align*}
     L^{ar}(f, k+t)^2= \sum_{l,\, m \geq 1}\frac{\lambda_f(l) \lambda_f(m)}{\sqrt{lm}} (lm)^{-t}.
 \end{align*}
 Now, since $L^{ar}(f,k)^2=|L ^{ar}(f, k)|^2$, we have
 \begin{align*}
     \Gamma(k)^2 \cdot |L^{ar}(f, k)|^2=& \mathrm{Res}_{t=0} \left(\frac{p}{4 \pi^2}\right)^t \Gamma(k+t)^2 L^{ar}(f, k+t)^2 \, \frac{1}{t}\\
     =& \frac{1}{\pi i} \int_{\mathrm{Re}(t)=\frac{3}{4}}\left(\frac{p}{4 \pi^2}\right)^t \Gamma(k+t)^2 L^{ar}(f, k+t)^2 \, \frac{dt}{t}\\
     =& \frac{1}{\pi}\sum_{l,\, m \geq 1}\frac{\lambda_f(l) \lambda_f(m)}{\sqrt{lm}} \int_{\mathrm{Re}(t)=\frac{3}{4}}\left(\frac{p}{4 \pi^2}\right)^t \frac{\Gamma(k+t)^2}{(lm)^t}  \, \frac{dt}{t}\\
     =& 2 \sum_{l,\, m \geq 1} G_k\left(\frac{lm}{p}\right)\frac{\lambda_f(l) \lambda_f(m)}{\sqrt{lm}}.
 \end{align*}
 This completes the proof.
 \end{proof}
 \begin{lem}\label{Gkbound}
     For $x\geq 1$ we have $G_k(x)\ll e^{-c\sqrt{x}}$ with a fixed constant $c>0$.
 \end{lem}
 \begin{proof}
     We shift the contour of the integral to $\mathrm{Re}(t)=\sqrt{x}$ and by using the Stirling formula from \cite{MR2061214}, i.e.,
 $\Gamma(s)= \sqrt{2\pi}\,\ s^{s-\frac{1}{2}}\,e^{-s}\left(1+O\left(\frac{1}{|s|}\right)\right)
 $, in the integral expression of $G_k(x)$, we get
 \begin{align*}
     G_k(x)\ll \int_{\mathrm{Re}(t)=\sqrt{x}}\frac{(k+t)^{2k+2t-1}}{(2\pi)^{2t}x^t} \, e^{-2(k+t)}\, \frac{dt}{t}\ll e^{-c\sqrt{x}}
 \end{align*}
 with a fixed constant $c>0$.
 \end{proof}
 \begin{rem}
 By shifting the contour to the left of the origin, one can easily show that the function $G_k(x)$ is bounded for $0<x<1$. 
 \end{rem}
\begin{prop}\label{proposition}
The Hecke operators $T_1, T_2, \ldots, T_D$ are linearly dependent on $z^{k-1} \otimes \lbrace0,  \infty\rbrace$, is equivalent to the existence of the non-trivial solutions to the equation
\begin{align*}
    \sum_{d_1,d_2<D}\frac{1}{\sqrt{d_1d_2}} \sum_{Id_1, Jd_2<D}\alpha_{Id_1}\bar{\alpha}_{Jd_2}\sum_{L, M\geq 1}\frac{G_k(LMd_1d_2\slash p)}{\sqrt{LM}}\sum_{f\in \mathcal{B}_{2k}(p)}\omega_f\lambda_f(IL)\lambda_f(JM)=0.
\end{align*}
\end{prop}
\begin{proof}
We have already seen that the linear dependency of $T_1, T_2, \ldots, T_D$ on $z^{k-1} \otimes \lbrace0,  \infty\rbrace$ is equivalent to the existence of a non-trivial solution of the following equation 
    \begin{align*}
         0=\sum_{i=1}^D\sum_{j=1}^D \alpha _i \bar{\alpha} _j \sum_{f\in \mathcal{B}_{2k}(p)} \omega_f \lambda_f(i) \lambda_f(j)\big|L^{ar}(f, k) \big|^2.
    \end{align*}
Using the value of $\big|L^{ar}(f, k) \big|^2$ from Lemma \ref{Lfunctionatk}, we get
    \begin{align*}
        0=\sum_{i=1}^D\sum_{j=1}^D \alpha _i \bar{\alpha} _j \sum_{l,\, m \geq 1} G_k\left(\frac{lm}{p}\right)\frac{1}{\sqrt{lm}} \sum_{f\in \mathcal{B}_{2k}(p)} \omega_f \lambda_f(i) \lambda_f(j)\lambda_f(l) \lambda_f(m).
    \end{align*}
    Substitute $i=Id_1,\, j=Jd_2$ where $d_1=(i,l),\, d_2=(j,m)$ and we also substitute $Ld_1=l,\, Md_2=m$. Then we get
    \begin{align*}
        0=\sum_{d_1,d_2<D}\frac{1}{\sqrt{d_1d_2}} \sum_{Id_1, Jd_2<D}\alpha_{Id_1}\bar{\alpha}_{Jd_2}&\sum_{L, M\geq 1}\frac{G_k(LMd_1d_2\slash p)}{\sqrt{LM}}\sum_{f\in \mathcal{B}_{2k}(p)}\omega_f\lambda_f(Id_1)\lambda_f(Jd_2)\lambda_f(Ld_1)\lambda_f(Md_2). 
    \end{align*}
    We use the following multiplicative relation of the eigenvalues of Hecke operatos 
    \begin{align*}
         \lambda_f(r) \lambda_f(s)=\sum_{d|(r,s)}\lambda_f\left(\frac{rs}{d}\right),
    \end{align*}
    where $(r,s)$ denotes the gcd of $r$ and $s$, and it has no common factor with the level. Since in our case, we have the prime level $p$, we can avoid this restriction. Hence we have
    \begin{align}\label{prop}
        0=\sum_{d_1,d_2<D}\frac{1}{\sqrt{d_1d_2}} \sum_{Id_1, Jd_2<D}\alpha_{Id_1}\bar{\alpha}_{Jd_2}\sum_{L, M\geq 1}\frac{G_k(LMd_1d_2\slash p)}{\sqrt{LM}} \sum_{f\in \mathcal{B}_{2k}(p)}\omega_f\lambda_f(IL)\lambda_f(JM).
    \end{align}
    This completes the proof.
    \end{proof}


  Now we use Lemma \ref{petersson} in the equation \eqref{prop}. Then by $S_{\mathrm{main}}$ we denote the term involving the Kronecker delta function and by $S_{\mathrm{off}}$ we denote the remaining terms involving the Kloosterman sum, i.e.,
  \[
  S_{\mathrm{main}}:= \sum_{d_1,d_2<D}\frac{1}{\sqrt{d_1d_2}} \sum_{Id_1, Jd_2<D}\alpha_{Id_1}\bar{\alpha}_{Jd_2}\sum_{L, M\geq 1}\frac{G_k(LMd_1d_2\slash p)}{\sqrt{LM}}\,\ \delta_{IL\, JM}
  \]
  and 
  \begin{align*}
     S_{\mathrm{off}}:=\sum_{d_1,\,d_2<D}\frac{1}{\sqrt{d_1d_2}} &\sum_{Id_1, \,Jd_2<D}\alpha_{Id_1}\bar{\alpha}_{Jd_2}\sum_{L, M\geq 1}\frac{G_k(LMd_1d_2\slash p)}{\sqrt{LM}} \\
     &\times\left(2 \pi i^{2k} \sum_{p|c}\frac{\mathcal{S}(IL, JM, c )}{c} \,J_{2k-1}\left(\frac{4\pi \sqrt{ILJM}}{c}\right)\right).
\end{align*}
    \section{Lower bound for the main term}
   In this section, we establish a lower bound for $S_{\mathrm{main}}$. From the previous section, we get the following:
    \begin{align*}
        S_{\mathrm{main}}=\sum_{d_1,d_2<D}\frac{1}{\sqrt{d_1d_2}} \sum_{Id_1, Jd_2<D}\alpha_{Id_1}\bar{\alpha}_{Jd_2}&\sum_{IL=JM}\frac{G_k(LMd_1d_2\slash p)}{\sqrt{LM}}.
    \end{align*}
   By employing the same change of variables as VanderKam \cite[p. 352]{MR1760688}, we obtain
    \begin{align*}
        S_{\mathrm{main}}=\sum_{L=1}^{\infty}L\sum_{U,\, V <\frac{D}{L}} \tau (U)\tau(V)\,x_{UL}\bar{x}_{VL}\sum_{T|L}\frac{\mu(T)}{T}\sum_{A=1}^{\infty}\frac{G_k(A^2T^2UV\slash p)}{A},
    \end{align*}
    where $x_c=\alpha_c\slash \sqrt{c}$ and $\tau(u)$ is the divisor function, i.e., $\tau(u):= \displaystyle\sum_{d|u}^{}  1$.
    \begin{lem}\label{sumGk}
        With the above notations, we have
        \begin{align*}
            \sum_{A=1}^{\infty}\frac{G_k(A^2\slash X)}{A}= \frac{\left((k-1)!\right)^2}{2}\log X+c_0+ O\left(\frac{1}{X}\right),
        \end{align*}
        where $c_0$ is a fixed constant.
    \end{lem}
    \begin{proof}
    From the integral representation of the function $G_k(x)$, we have
    \begin{align*}
         \sum_{A=1}^{\infty}\frac{G_k(A^2\slash X)}{A}&=\frac{1}{2\pi i}\sum_{A=1}^{\infty} \int_{\mathrm{Re}(t)=\frac{3}{4}}\frac{\Gamma(k+t)^2}{(2\pi)^{2t}}\frac{X^t}{A^{2t+1}}\frac{dt}{t}\\
         &=\frac{1}{2\pi i}\int_{\mathrm{Re}(t)=\frac{3}{4}}\frac{\Gamma(k+t)^2}{(2\pi)^{2t}}X^t\, \zeta(2t+1)\frac{dt}{t}.
    \end{align*}
    Now by the contour shift, for any $\epsilon>0$ we have 
    \begin{align*}
        \frac{1}{2\pi i}\int_{\mathrm{Re}(t)=\frac{3}{4}}\frac{\Gamma(k+t)^2}{(2\pi)^{2t}}X^t\, \zeta(2t+1)\frac{dt}{t}=&\frac{1}{2\pi i}\int_{\mathrm{Re}(t)=-k-\epsilon}\frac{\Gamma(k+t)^2}{(2\pi)^{2t}}X^t\zeta(2t+1)\frac{dt}{t}\\ & + \mathrm{Res}_{t=-k}\left(\frac{\Gamma(k+t)^2}{(2\pi)^{2t}}\frac{X^t\zeta(2t+1)}{t}\right)\\ & +\mathrm{Res}_{t=0}\left(\frac{\Gamma(k+t)^2}{(2\pi)^{2t}}\frac{X^t\zeta(2t+1)}{t}\right).
    \end{align*}
   Using the residue formula, we compute
    \begin{align*}
        \mathrm{Res}_{t=0}\left(\frac{\Gamma(k+t)^2}{(2\pi)^{2t}}\frac{X^t\zeta(2t+1)}{t}\right)&= \lim_{t \to 0}\left(\frac{d}{dt}\left(\frac{t^2\Gamma(k+t)^2}{(2\pi)^{2t}}\frac{X^t\zeta(2t+1)}{t}\right)\right)\\
        &=\frac{\left((k-1)!\right)^2}{2}\log X+c_0;
    \end{align*}
    with a fixed $c_0$, and this completes the proof. 
    \end{proof}
    \begin{prop}
   Let $\phi(L)$ be the Euler function and $y_L=\displaystyle\sum_{UL<D}\tau (U)\, x_{UL}$.
    For given $\delta>0$ there exists a constant $c_\delta$ such that 
        \begin{align*}
            S_{\mathrm{main}}\gg \log p \sum_{L\geq 1}\phi(L) |y_L|^2 \quad \text{for} \quad c_\delta D^2< p^{1-\delta}.
            \end{align*}
     \end{prop}
    \begin{proof}
    Recall that
    \begin{align*}
        S_{\mathrm{main}}=\sum_{L=1}^{\infty}L\sum_{U,\, V <\frac{D}{L}} \tau (U)\tau(V)\,x_{UL}\bar{x}_{VL}\sum_{T|L}\frac{\mu(T)}{T}\sum_{A=1}^{\infty}\frac{G_k(A^2T^2UV\slash p)}{A},
    \end{align*}
        Then by using Lemma \ref{sumGk} and the computations from VanderKam \cite[pp. 353 -- 354]{MR1760688}, we get the following:
\begin{align*}
    S_{\mathrm{main}}\geq \sum_{L=1}^{\infty}\phi(L) |y_L|^2 \left(\log \left(\frac{p^{\frac{((k-1)!)^2}{2}}}{D}\right)+c_0-c_1\right),
\end{align*}
where $y_L=\displaystyle\sum_{UL<D}\tau (U)\, x_{UL}$. 
Recall  the function $\Lambda$ as in VanderKam \cite[p. 353]{MR1760688} and  $c_1$ is a fixed quantity obtained by prime number theorem :
\[
\sum\limits_{w < \frac{D}{L}} \frac{\Lambda(w)}{w} < \log\bigg(\frac{D}{L}\bigg) +c_1.
\]
For $D^2< c_{\delta} p^{1-\delta}$, we have 
\begin{align*}
    S_{\mathrm{main}}\geq \sum_{L=1}^{\infty}\phi(L) |y_L|^2 \left(\bigg(\frac{((k-1)!)^2}{2}-\frac{1}{2}+\frac{\delta}{2}\bigg)\log(p)+c_0-c_1-\frac{1}{2} \log (c_{\delta})\right). 
\end{align*}
 We then have
\begin{align*}
     S_{\mathrm{main}}\geq \sum_{L=1}^{\infty}\phi(L) |y_L|^2 \bigg({\frac{((k-1)!)^2}{2}}-\frac{1}{2}+\frac{\delta}{2}\bigg)\log p+\sum_{L=1}^{\infty}\phi(L) |y_L|^2\bigg(- \frac{1}{2} \log (c_{\delta})+(c_0-c_1)\bigg).
\end{align*}
Choose $c_{\delta}$ such that  $-\frac{1}{2} \log (c_{\delta}) +(c_0-c_1)>0$.
\begin{align*}
     S_{\mathrm{main}}\geq \sum_{L=1}^{\infty}\phi(L) |y_L|^2 \bigg({\frac{((k-1)!)^2}{2}}-\frac{1}{2}+\frac{\delta}{2}\bigg)\log p.
\end{align*}
Since $c_{\delta}$ is a fixed quantity, it is clear that the first term dominates the second term. 
Since $C_{k, \delta}:=\frac{((k-1)!)^2}{2}-\frac{1}{2}+\frac{\delta}{2}>1$, we get
\begin{align*}
     S_{\mathrm{main}}> C_{k, \delta} \log p\sum_{L=1}^{\infty}\phi(L) |y_L|^2. 
\end{align*}
We deduce that $S_{\mathrm{main}} \gg \log p(\sum\limits_{L=1}^{\infty}\phi(L) |y_L|^2)$;
where the implied constant depends on $k$ and $\delta$.
    \end{proof}

\section{Upper bound of Kloosterman sum part}
Here we consider the part which involves the Kloosterman sums, i.e., $S_{\mathrm{off}}$. Recall that
\begin{align*}
     S_{\mathrm{off}}=\sum_{d_1,\,d_2<D}\frac{1}{\sqrt{d_1d_2}} &\sum_{Id_1, \,Jd_2<D}\alpha_{Id_1}\bar{\alpha}_{Jd_2}\sum_{L, M\geq 1}\frac{G_k(LMd_1d_2\slash p)}{\sqrt{LM}} \\
     &\times\left(2 \pi i^{2k} \sum_{p|c}\frac{\mathcal{S}(IL, JM, c )}{c} \,J_{2k-1}\left(\frac{4\pi \sqrt{ILJM}}{c}\right)\right).
\end{align*}
Making the substitution $x_{Id_1}= \frac{\alpha_{Id_1}}{\sqrt{Id_1}}$ and $\bar{x}_{Jd_2}= \frac{\alpha_{Jd_2}}{\sqrt{Jd_2}}$, we get
\begin{align*}
     S_{\mathrm{off}}=\sum_{d_1,\,d_2<D} \,\, \,\ &\sum_{Id_1, \,Jd_2<D}x_{Id_1}\bar{x}_{Jd_2}\sqrt{IJ}\sum_{L, M\geq 1}\frac{G_k(LMd_1d_2\slash p)}{\sqrt{LM}} \\
     &\times\left(2 \pi i^{2k} \sum_{p|c}\frac{\mathcal{S}(IL, JM, c )}{c} \,J_{2k-1}\left(\frac{4\pi \sqrt{ILJM}}{c}\right)\right).
\end{align*}
For notational convenience we write $S_{\mathrm{off}}= S_{\mathrm{off}}^{c<p^2}+S_{\mathrm{off}}^{c\geq p^2}$, where we define
\begin{align}\label{SoffclessthanNsquare}
S_{\mathrm{off}}^{c<p^2}:=\sum_{d_1,\,d_2<D} \,\, \,\ &\sum_{Id_1, \,Jd_2<D}x_{Id_1}\bar{x}_{Jd_2}\sqrt{IJ}\sum_{L, M\geq 1}\frac{G_k(LMd_1d_2\slash p)}{\sqrt{LM}} \notag\\
     &\times\left(2 \pi i^{2k} \sum_{\substack{p|c\\ c <p^2}}\frac{\mathcal{S}(IL, JM, c )}{c} \,J_{2k-1}\left(\frac{4\pi \sqrt{ILJM}}{c}\right)\right)
\end{align}
and
\begin{align}\label{SoffcgraterthanNsquare}
    S_{\mathrm{off}}^{c\geq p^2}:=\sum_{d_1,\,d_2<D} \,\, \,\ &\sum_{Id_1, \,Jd_2<D}x_{Id_1}\bar{x}_{Jd_2}\sqrt{IJ}\sum_{L, M\geq 1}\frac{G_k(LMd_1d_2\slash p)}{\sqrt{LM}} \notag\\
     &\times\left(2 \pi i^{2k} \sum_{\substack{p|c\\ c \geq p^2}}\frac{\mathcal{S}(IL, JM, c )}{c} \,J_{2k-1}\left(\frac{4\pi \sqrt{ILJM}}{c}\right)\right).
\end{align}
\begin{lem}
For $D^2< c_{\delta}p^{1-\delta}$,  we have the following upper bound: 
    \begin{align*}
         S_{\mathrm{off}}^{c\geq p^2}\ll p^{-1+\epsilon} \sum_{L\geq 1}\phi(L) |y_L|^2
    \end{align*}
    for any given $\epsilon>0$. 
\end{lem}
\begin{proof}
    We recall the Weil's bound for the Kloosterman sum, which says
    \begin{align}\label{weilbound}
        \mathcal{S}(IL, JM, c ) \leq (IL, JM, c)^{\frac{1}{2}}\,\tau(c)\, c^{\frac{1}{2}}.
    \end{align}
    Also, recall that the $J$-Bessel function satisfies the bound
    \begin{align}\label{Jbound}
        J_{2k-1}\left(\frac{4\pi \sqrt{ILJM}}{c}\right)\ll \left(\frac{4\pi \sqrt{IJLM}}{c}\right).
    \end{align}
Using \eqref{weilbound} and \eqref{Jbound} in the expression \eqref{SoffcgraterthanNsquare}, we get
\begin{align*}
    S_{\mathrm{off}}^{c\geq p^2}\ll \sum_{d_1,\,d_2<D} \,\, \, \sum_{Id_1, \,Jd_2<D}|x_{Id_1}|\,|\bar{x}_{Jd_2}|\,IJ\sum_{L, M\geq 1}{G_k(LMd_1d_2\slash p)} 
      \sum_{\substack{p|c\\ c \geq p^2}}c^{-\frac{3}{2}}\, (IL, JM, c)^{\frac{1}{2}}\,\tau(c).
\end{align*}
Now using the fact that $G_k(x)\ll e^{-c\sqrt{x}}$ (see Lemma \ref{Gkbound}), we have
\begin{align*}
S_{\mathrm{off}}^{c\geq p^2}\ll p^{-1+\epsilon_1}\sum_{d_1,\,d_2<D} \,\, \, \sum_{Id_1, \,Jd_2<D}|x_{Id_1}|\,|\bar{x}_{Jd_2}|\,IJ
\end{align*}
for any given $\epsilon_1>0$. Then using the computation from VanderKam \cite[p. 354]{MR1760688} we get
\begin{align*}
    S_{\mathrm{off}}^{c\geq p^2}\ll p^{-1+\epsilon_1}D^{\epsilon_2} \sum_{L\geq 1}\phi(L) |y_L|^2
\end{align*}
for any given $\epsilon_2>0$. Now, for $D^2<c_{delta} p^{1-\delta}$ we get
\begin{align*}
    S_{\mathrm{off}}^{c\geq p^2}\ll p^{-1+\epsilon} \sum_{L\geq 1}\phi(L) |y_L|^2
\end{align*}
for some $\epsilon>0$, and this completes the proof.
\end{proof}

\begin{lem}
    With the above notation, for $D^2< c_{\delta} p^{1-\delta}$ we have
    \begin{align*}
        S_{\mathrm{off}}^{c<p^2}\ll p^{\epsilon}\sum_{L\geq 1}\phi(L) |y_L|^2
    \end{align*}
    for any given $\epsilon>0$.
\end{lem}
\begin{proof}
    Without loss of generality, we assume that the sums on $L$ and $M$ are over dyadic blocks of size $\mathscr{L}$, $\mathscr{M}$, respectively with $\mathscr{LM}\ll \frac{p^{1+\epsilon}}{d_1d_2}$. 

From the definition of Kloosterman sum, we recall that (see \cite[p. 55]{MR1474964})
\begin{align}\label{eq1}
\mathcal{S}(IL, JM, c )= \sum_{(a,c)=1}e\left(\frac{aIL+\bar{a}JM}{c}\right),
\end{align}
where $e(x):= e^{2\pi i x}$ and $a\bar{a}\equiv 1\,(\mathrm{mod}\, c)$. The $J$-Bessel function can be written as (see \cite[p. 50]{MR1474964}) 
\begin{align}\label{eq2}
    J_{2k-1}\left(\frac{4\pi \sqrt{ILJM}}{c}\right)=\sum_{\ell=0}^{\infty}\frac{(-1)^\ell}{\ell! \,\ \Gamma(\ell+2k)}\left( \frac{2\pi \sqrt{IJLM}}{c}\right)^{2k+2\ell-1}.
\end{align}
Now, using \eqref{eq1} and \eqref{eq2} in the expression \eqref{SoffclessthanNsquare}, we get
\begin{align*}
    S_{\mathrm{off}}^{c<p^2}= &\int_{\mathrm{Re}(t)=\frac{3}{4}}\frac{i^{2k-1}\Gamma(k+t)^2}{t}\sum_{\ell=0}^{\infty} a_{\ell}\, p^t\sum_{d_1,\,d_2<D} \,\, \,\ \sum_{Id_1, \,Jd_2<D}x_{Id_1}\bar{x}_{Jd_2}(IJ)^{k+\ell}\\
    &\times(d_1d_2)^{-t}\sum_{\substack{p|c\\ c <p^2}}\frac{1}{c^{2k+2\ell}}\sum_{(a,c)=1}\,\ \sum_{L,\,M}(LM)^{k+\ell-1-t}\,\ e\left(\frac{aIL+\bar{a}JM}{c}\right) dt,
\end{align*}
where $a_\ell=\frac{(-1)^\ell}{\ell! \,\ \Gamma(\ell+2k)}$. Then by using VanderKam \cite[Lemma 3.3 and Lemma 3.4]{MR1760688}, we can write the inner sum as
\begin{align*}
    &\sum_{(a,c)=1}\,\ \sum_{L,\,M}(LM)^{k+\ell-1-t}\,\ e\left(\frac{aIL+\bar{a}JM}{c}\right)\\
    &=\sum_{(a,c)=1}\,\ \sum_{L=\mathscr{L}}^{2\mathscr{L}}L^{k+\ell-1-t}\,\ e\left(\frac{aIL}{c}\right)\sum_{M=\mathscr{M}}^{2\mathscr{M}}M^{k+\ell-1-t}\,\ e\left(\frac{\bar{a}JM}{c}\right)\\
    &\ll (\mathscr{LM})^{k+\ell-1-\mathrm{Re}(t)}\sum_{(a,c)=1} \mathrm{min}\left( \mathscr{L}, \frac{c^{\epsilon} (1+ |k+\ell -1-t|)}{\lfloor aI/c \rfloor}\right)\mathrm{min}\left( \mathscr{M}, \frac{c^{\epsilon} (1+ |k+\ell -1-t|)}{\lfloor \bar{a}J/c \rfloor}\right)\\
    &\ll (\mathscr{LM})^{k+\ell-1-\mathrm{Re}(t)} \left(1+ |k+\ell -1-t|\right)^2(\mathscr{LM}+c)c^\epsilon\\
    &\ll \left(\frac{p^{1+\epsilon}}{d_1d_2}\right)^{k+\ell -1 -\mathrm{Re}(t) } \left(\frac{p^{1+\epsilon}}{d_1d_2}+c\right) c^\epsilon  \left(1+ |k+\ell -1-t|\right)^2
\end{align*}
for any $\epsilon>0$. In the third line $\lfloor aI/ c \rfloor$ denotes the distance from $aI/c$ to the nearest integer.
Hence we get
\begin{align}\label{boundclessNsquare}
   S_{\mathrm{off}}^{c<p^2}& \ll \int_{\mathrm{Re}(t)=\frac{3}{4}}\frac{i^{2k-1}\Gamma(k+t)^2}{t}\sum_{\ell=0}^{\infty} a_{\ell}\, p^t\sum_{d_1,\,d_2<D} \,\, \,\ \sum_{Id_1, \,Jd_2<D}x_{Id_1}\bar{x}_{Jd_2}(IJ)^{k+\ell}\notag \\
    &\times(d_1d_2)^{-t}\sum_{\substack{p|c\\ c <p^2}}\frac{1}{c^{2k+2\ell}} \left(\frac{p^{1+\epsilon}}{d_1d_2}\right)^{k+\ell -1 -\mathrm{Re}(t) } \left(\frac{p^{1+\epsilon}}{d_1d_2}+c\right) c^\epsilon  \left(1+ |k+\ell -1-t|\right)^2 dt.
\end{align}
For each $\ell$ and $t$ the right hand side of \eqref{boundclessNsquare} is bounded by $\Gamma(k+t)^2\left(1+ |k+\ell -1-t|\right)^2$ times the term
\begin{align}\label{boundforeachlandt}
    p^{-k-\ell +\epsilon} D^\epsilon\left( \sum_{L<D}|y_L|L^{\ell+k}\right)^2&\ll  p^{-k-\ell +\epsilon} D^\epsilon\left( \sum_{L<D}\phi(L)|y_L|^2\right) \left(\sum_{L<D}\frac{L^{2\ell+2k}}{\phi(L)}\right)\notag\\&\ll
    (pD)^\epsilon\left(\frac{D^2}{p}\right)^{k+\ell} \left( \sum_{L<D}\phi(L)|y_L|^2\right).
\end{align}
For $D^2<c_{\delta} p^{1-\delta}$, the right hand side of \eqref{boundforeachlandt} is bounded by $p^{\frac{3\epsilon}{2}}p^{-\delta(\ell+k +\frac{\epsilon}{2})}\left( \sum_{L<D}\phi(L)|y_L|^2\right)$, and by taking the sum over $\ell$ and integration over $\mathrm{Re}(t)=3/4$, we see that the term $S_{\mathrm{off}}^{c<p^2}$ is bounded by $p^{\epsilon}\displaystyle \sum_{L}\phi(L) |y_L|^2$ for any $\epsilon>0$.
\end{proof}
\section{Generalization of VanderKam's theorem for higher weights}
\begin{thm}\label{mainthm1}
Let $k\geq 1$ be a fixed positive integer. For any integer $D  \geq 4$. For given $\delta>0$ there exists a constant $c_\delta$ such that the  Hecke operators $T_1, T_2, \ldots, T_D$ act linearly independently on the cycle $z^{k-1}\otimes \mathbf{e}$ within the space of cuspidal symbols $\mathbb{S}_{2k}(\Gamma_0(p)) \otimes \mathbb{R}$ for all primes $D^2<c_{\delta} p^{1-\delta}$.
\end{thm}

\begin{proof}
The linear dependency of the Hecke operators $T_1, T_2, \ldots, T_D$ is equivalent to saying there exists a non-trivial solution $\alpha_1, \ldots, \alpha_D$ of  
  \begin{align}\label{eq}
         \sum_{i=1}^D\sum_{j=1}^D \alpha _i \bar{\alpha} _j \sum_{f\in \mathcal{B}_{2k}(p)} \omega_f \lambda_f(i) \lambda_f(j)\big|L^{ar}(f, k) \big|^2=0.
    \end{align}
This means that the Hecke operators $T_1, T_2, \ldots, T_D$ act linearly independently on $z^{k-1} \otimes\lbrace0,  \infty\rbrace$ if and only if there does not exist any non-trivial solution of the equation \eqref{eq}. With our notations the we have
\begin{align*}
   \sum_{i=1}^D\sum_{j=1}^D \alpha _i \bar{\alpha} _j \sum_{f\in \mathcal{B}_{2k}(p)} \omega_f \lambda_f(i) \lambda_f(j)\big|L^{ar}(f, k) \big|^2 &= \sum_{i=1}^D\sum_{j=1}^D \alpha _i \bar{\alpha} _j a_{ij}\\
   &=Q(\alpha_i,\bar{\alpha_j})\\
   &=S_{\mathrm{off}}+S_{\mathrm{main}},
\end{align*}
where $a_{ij}:=\sum\limits_{f\in \mathcal{B}_{2k}(p)}^h \lambda_f(i) \lambda_f(j)\big|L^{ar}(f, k) \big|^2$ with $\sum\limits^h$ denoting the harmonic average.
We form the quadratic form $Q(\underline{x},\underline{y})=\sum_{i=1}^D\sum_{j=1}^D x _i x _j a_{ij} $with $\underline{x}=
(x_i)_{1 \leq i \leq D}$ and $\underline{y}=
(y_i)_{1 \leq i \leq D}$.

 Now, note that the off-diagonal term $S_{\mathrm{off}}$ is not large enough to cancel the main diagonal term $S_{\mathrm{main}}$. This implies, for $c_\delta D^2<p^{1-\delta}$ there is no non-trivial solutions $\alpha_1, \ldots, \alpha_D$ of the equation \eqref{eq}. Hence, the Hecke operators $T_1, T_2, \ldots, T_D$ are linearly independent on the cycle $z^{k-1} \otimes\lbrace0, \infty\rbrace$. 
\end{proof}
\section{Main Theorem}
  We know that the Galois group $\mathrm{Gal}(\overline{\mathbb{Q}}/\mathbb{Q})$ acts on $\mathcal{B}_{2k}(p)$. For $f \in \mathcal{B}_{2k}(p)$,  let $[f]$ be the set of all orbits of $f$. 
    Let $K_{[f]}$ be the Hecke field associated to an eigenform $f \in \mathcal{B}_{2k}(p)$. Let $\mathrm{Ann}_{[f]}$ be the annihilator ideal of $f$ in $\mathbb{T}$. Let $\mathcal{E}_{k-1}$ be the set of orbits such that $L^{ar}(f,k)\neq 0$.
We start by showing the following lemma:
\begin{lem}\label{identity}
Let $I_e:=\mathrm{Ann}(z^{k-1} \otimes \mathbf{e}):= \lbrace T\in \mathbb{T}\mid T(z^{k-1} \otimes \mathbf{e})=0 \rbrace
    $ be the annihilator ideal of $z^{k-1} \otimes \mathbf{e}$. Then we have
    \[
        \displaystyle\bigcap_{[f]\in \mathcal{E}_{k-1}}\mathrm{Ann}({[f]})=\displaystyle\mathrm{Ann}\left(z^{k-1} \otimes \mathbf{e}\right).
  \]
    \end{lem}
    \begin{proof}
Let us consider an element $T$ from the left-hand side. Then
 for every $f$ belonging to the set $\mathcal{B}_{2k}(p)$ such that $L^{ar}(f, k)\neq0$, we have the following:
   \begin{align}
   \label{Heckee}
       \langle T(z^{k-1} \otimes \mathbf{e}), f \rangle= \langle z^{k-1} \otimes \mathbf{e}, Tf \rangle=0.
   \end{align}
  If $f \in \mathcal{B}_{2k}(p)$ be such that $L^{ar}(f, k)=0$ then from \eqref{lfun} we get 
    \begin{align*}
        \langle z^{k-1} \otimes \mathbf{e}, f \rangle=0.
    \end{align*}
    Therefore $\langle T(z^{k-1} \otimes \mathbf{e}), f \rangle=0$ because $f\in \mathcal{B}_{2k}(p)$. Hence, if $T$ be an element in left hand side  of \eqref{identity} then
    $Tf=0$ (since it is in the annihilator of $f$)  and $L^{ar}(f,k)\neq 0$ as $f \in \mathcal{E}_{k-1}$. 
    For all $f \in \mathcal{B}_{2k}(p)$, we have by \eqref{Heckee}, i.e., $ \langle T(z^{k-1} \otimes \mathbf{e}), f \rangle=0$. We deduce that 
    $T(z^{k-1} \otimes \mathbf{e} )=0$. This implies that the Hecke operator $T\in \displaystyle\mathrm{Ann}(z^{k-1} \otimes \mathbf{e})$.
    
    Now, for the other inclusion, let us consider an element $T\in \displaystyle\mathrm{Ann}(z^{k-1} \otimes \mathbf{e})$ and
    hence $ T( z^{k-1} \otimes \mathbf{e})=0$. 
    Let $f\in \mathcal{E}_{k-1}$ and hence $f \in  \mathcal{B}_{2k}(p)$ (Hecke basis) such that $L^{ar}(f, k)\neq0$. If possible $T f \neq 0$.  Since it is a Hecke basis so $Tf=\lambda_f(T) f$.  Since $T(f) \neq 0$, we have  $\lambda_f(T_r) \neq 0$. We then have
     \begin{align*}
       \langle z^{k-1} \otimes \mathbf{e}, T f \rangle= \langle T(z^{k-1} \otimes \mathbf{e}), f \rangle=0.
   \end{align*}
   Now by \eqref{lfun} implies $\lambda_f(T)\langle z^{k-1} \otimes \mathbf{e}, f \rangle=\lambda_f(T) L^{ar}(f, k) =0$. This is a clear contradiction. 
   Therefore $T$ is an element in the left hand side. This completes the proof.
   \end{proof}
Using Lemma \ref{identity} we prove the following lemma which the number field analogue of \cite[Lemma 7.12]{MR3530536}, this is true even for the function field setting.
\begin{lem}\label{keylemma}
    The $\Z$-module $\mathbb{T}(z^{k-1} \otimes\mathbf{e}) \otimes \mathbb{Q}$ is a $\mathbb{Q}$  vector space of  dimension $\#\lbrace f\in \mathcal{B}_{2k}(p)\mid L^{ar}(f,k)\neq 0\rbrace.$
\end{lem}
\begin{proof}
We have the following isomorphism 
    $$\frac{\mathbb{T}}{ \mathrm{Ann}(z^{k-1} \otimes \mathbf{e})}\cong\mathbb{T}(z^{k-1} \otimes \mathbf{e}). $$
 We calculate the rank of the quotient $\mathbb{T}\slash \mathrm{Ann}(z^{k-1} \otimes \mathbf{e})$. 
Now using Lemma \ref{identity} we can write that $\displaystyle\mathrm{Ann}(z^{k-1} \otimes \mathbf{e})$ is the kernel of the  $\Z$-modules homomorphism 
   \begin{align*}
       \varphi: \mathbb{T}\to \prod_{[f]\in \mathcal{E}_{k-1}} \frac{\mathbb{T}}{\mathrm{Ann}_{[f]}} \simeq \frac{\mathbb{T}}{\bigcap_{[f]\in \mathcal{E}_{k-1}}\mathrm{Ann}_{[f]}} \simeq \frac{\mathbb{T}}{\mathrm{Ann}\left(z^{k-1} \otimes \mathbf{e}\right)}.
   \end{align*}
   So, the $\mathbb{Q}$-vector space $\frac{\mathbb{T}}{\displaystyle\mathrm{Ann}(z^{k-1} \otimes \mathbf{e})} \otimes_{\Z}\mathbb{Q}$ is isomorphic to $\displaystyle\prod_{[f]\in \mathcal{E}_{k-1}} K_{[f]}$ and the dimension of the same is 
   $$\# \lbrace f\in \mathcal{B}_{2k}(p)\mid L^{ar}(f,k)\neq 0\rbrace.$$
   This completes the proof.

   \end{proof}
   

           \begin{rem}
      If we restrict ourselves to the subset $\mathcal{B}^+_{2k}(p)$ consisting of modular forms with root numbers $\epsilon(f)=1$. Iwaniec--Sarnak \cite[p. 158]{MR1815374} shows that 
for almost half of them, corresponding arithmetic $L$-functions are nonvanishing at the {\it central} critical point $k$. This was the conjecture of Katz--Sarnak \cite{MR1640151}. In our present paper, we restrict our attention to 
$\mathcal{B}_{2k}(p)$ rather than on smaller vector space $\mathcal{B}^+_{2k}(p)$.
   \end{rem}

  \begin{proof}[Proof of \ref{mainthm} $(a)$]
      By Theorem~\ref{mainthm1}, we have $T_1 z^{k-1} \otimes \mathbf{e},\ldots, T_D z^{k-1} \otimes \mathbf{e}$ are linearly independent over $\Z$ for $D<c_{\delta}p^{\frac{1}{2}-\delta}$.
      We then have the rank of $\frac{\T}{I_e}$ is $D$. We deduce there are $D$ many elements in Hecke basis with non-vanishing central critical values by Lemma~\ref{keylemma}. 
      \end{proof}

\subsection{Algebraic modular form}
Let $N \in \N$ be a natural number. By $M_{2k}(\Ga_0(N))$, we denote the $\C$-vector space of classical modular forms of weight $2k$ and level $N$ as in \cite{MR2112196}. Inside this $\C$-vector space, let  $S_{2k}(\Ga_0(N))$ be the space of cusps forms.
 
Let $M_{2k}(\Ga_0(N);\Z)$ denote the lattice consisting of elements $ M_{2k}(\Ga_0(N))$ whose Fourier coefficients at the cusp $\infty$  belong to  $\Z$. By \cite[Corollary 12.3.12 and Proposition 12.4.1]{MR1357209}, we know that $M_{2k}(\Ga_0(N); \Z)$ is stable under the action of Hecke operators and contains a basis of $M_{2k}(\Ga_0(N))$.
 For any $\Z$-algebra  $S$, we define $M_{2k}(\Ga_0(N); S) := M_{2k}(\Ga_0(N); \Z) \otimes S$. By $S_{2k}(\Ga_0(N);S)$ we denote the corresponding space of cusp forms (the forms for which constant terms are zero).

Let $\T:= \T(N) \subseteq \mathrm{End}_{\Z}(S_{2k}(\Ga_0(N); \Z)$
be the Hecke algebra generated by the Hecke operators $T_n$ for all $n \in \N$. We recall the following fundamental proposition due to Hida (cf. Diamond--Im~\cite[Proposition 12.4.13, p. 118]{MR1357209}): 
\begin{prop}
\label{Hida}
For any $\Z$-algebra $S$, we define $\T_S:= \T\otimes S$. The pairing: $S_{2k}(\Gamma_0(N); S) \times \T_S \rightarrow  S$
given by $(f, T_l) \mapsto a_1(T_l f)$  is  a perfect pairing of $S$-modules.  In particular, we deduce  that $S_{2k}(\Gamma_0(N); S) \simeq \Hom_S(\T_S, S)$. 
\end{prop}
 From loc. cit., we have $\T$ is a free $\Z$-module of rank $r$, where $r= \mathrm{dim}_\C S_{2k}(\Ga_0(N))$. 
\subsection{Proof of the main theorem}
 The following proposition is a generalization to weight $2k$ from weight $2$, which is a theorem due to Kamienny (cf. \cite[Theorem 5]{MR2498054}): 

 \begin{prop}
 \label{equiv}
 For all prime numbers $l\neq p$, the following conditions are equivalent:
 \begin{enumerate}
\item 
There exists weight $2k$ cusp forms $\bar{f}_1, \bar{f}_2,\ldots,\bar{f}_D \in S_{2k}(\Ga_0(N),\F_l)$ with $I_e \bar{f}_i=0$ such that $(a_1(\bar{f}_i), a_2(\bar{f}_i), \ldots, a_D(\bar{f}_i))_{i \in \{1,\ldots, D\}}$
are linearly independent.

\item 
The images of $T_1,\ldots, T_D$ in $\frac{\T}{(l, I_e)}$ are $\F_l$-linearly independent.
\item 
The quantities $T_1 z^{k-1} \otimes \mathbf{e},\ldots, T_D z^{k-1} \otimes \mathbf{e} $ are linearly independent over $\F_l$. 
 \end{enumerate}
 \end{prop}
 \begin{proof}
 We first prove $(1)$ and $(2)$ are equivalent. 
 There is an induced {\it perfect} pairing between $\F_l$ vector spaces:
\[
S_{2k}(\Ga_0(N),\F_l)[I_e] \times \frac{\T}{(l, I_e)} \rightarrow \F_l.
\]
We deduce that  $\mathrm{dim}_{\F_l} S_{2k}(\Ga_0(N),\F_l)[I_e]=D$ if $\mathrm{dim}_{\F_l} \frac{\T}{(l,I_e)}=D$. 
From the elementary linear algebra, $\mathrm{dim}_{\F_l} S_{2k}(\Ga_0(N),\F_l)[I_e]=D$ implies there exist $\F_l$ linearly independent cusp forms $\lbrace \bar{f}_i \rbrace_{i=1, \ldots, D}$ in the sense of Serre  such that $T \bar{f}_i =0$ for all $T \in I_e$. We deduce that  $\bar{f}_i \in S_{2k}(\Ga_0(N),\F_l)$ are $\F_l$ linearly independent and $I_e \bar{f}_i=0$ for all $i \in \{1, 2, \ldots, D\}$. If we have a 
relation 
\[
c_1 \bar{f}_1+c_2\bar{f}_2+\cdots+c_D \bar{f}_D=0;
\]
this implies all the $c_i$'s are zero. For each $\bar{f}_i$ we have Fourier expansion of the form $f_i=\sum\limits_{n\geq 1} a_n(\bar{f}_i) q^n$ with $i=1, \ldots, D$. We have a system of equation of the form 
$A\underline{c}=0$ with $\underline{c}=(c_1, c_2, \ldots, c_D)$ where 
$$
A=\begin{pmatrix}
a_1(\bar{f}_1) & a_2(\bar{f}_2) &\cdots &  a_1(\bar{f}_D)\\
\vdots  & \vdots  & \ddots & \vdots\\
a_D(\bar{f}_1) & a_D(\bar{f}_2) & \cdots & a_D(\bar{f}_D)
\end{pmatrix}.
$$ Now, $\underline{c}=0$ implies vectors
$(a_1(\bar{f}_i), a_2(\bar{f}_i), \ldots, a_D(\bar{f}_i))_{i=1, \ldots, D}$ are linearly independent.

To verify the condition $(2)$ and $(3)$ are equivalent we observe that $\frac{\T}{I_e} \simeq \T (z^{k-1} \otimes \mathbf{e})$. Hence, we have 
\[
\frac{\T}{I_e} \otimes \F_l\simeq \T (z^{k-1} \otimes \mathbf{e}) \otimes \F_l.
\]
This completes the proof.
 \end{proof}
  \begin{proof}[Proof of \ref{mainthm} $(b)$]
      
     By Theorem~\ref{mainthm1}, we have $T_1 z^{k-1} \otimes \mathbf{e},\ldots, T_D z^{k-1} \otimes \mathbf{e}$ are linearly independent over $\Z$ for $D<c_{\delta}p^{\frac{1}{2}-\delta}$.
      We then have the rank of $\frac{\T}{I_e}$ is $D$. We deduce there are $D$ many elements in Hecke basis with non-vanishing central critical values by Lemma~\ref{keylemma}.  The dimension of the $\F_l$ vector space $\frac{\T}{I_e} \otimes \F_l$ is also $D$  for all prime numbers $l \neq p$ by tensoring over the field $\F_l$. Hence, for $D<c_{\delta}p^{\frac{1}{2}-\delta}$ we get $T_1 z^{k-1} \otimes \mathbf{e},\ldots, T_D z^{k-1} \otimes \mathbf{e} $ are linearly indepedent over $\F_l $ for all prime numbers $l \neq p$. 
   We now deduce our main theorem from Proposition~\ref{equiv}.
     \end{proof}
     
\section{Numerical computation using {\tt MAGMA} and the constant $c_{\delta}$}
If we use {\tt MAGMA} for $N=23$ and even weight two. Observe that $4<\sqrt{23}$. However, we provide the following {\tt MAGMA} computations:
\begin{align*}
&\mathtt{M:= ModularSymbols(23);
T2 := HeckeOperator(M,2);
T3 := HeckeOperator(M,3);
}\\
&
\mathtt{
T4 := HeckeOperator(M,4);
T5 := HeckeOperator(M,5);
E :=WindingElement(M);}
\\
&
\mathtt
{E;
-1*\{oo, 0\}\quad
E*T2;
-1*\{-1/11, 0\} + -3*\{oo, 0\}}
\\
&
\mathtt
{ E*T3;                  
\{-1/19, 0\} + -1*\{-1/17, 0\} + -1*\{-1/11, 0\} + -4*\{oo, 0\}
}\\
&
\mathtt{E*T4;                    
\{-1/19, 0\} + -1*\{-1/17, 0\} + -2*\{-1/11, 0\} + -7*\{oo, 0\}}
\\
&
\mathtt{E*T5;-2*\{-1/11, 0\} + -6*\{oo, 0\}}.
\end{align*}

It is easy to see that  $c_1 T_1 \mathbf{e}+c_2 T_2 \mathbf{e}+c_3 T_3 \mathbf{e} +c_4 T_4 \mathbf{e}=0$ will produce a clear dependence. That just shows that the constant in the theorem is not $1$. We did the computation even for weight $4$ and higher weights. 
Generally, we are getting linear dependence. However, for  $d=4$ and for primes $p \ge 257$, we are getting linear independence (cf.~\cite{Debarghawebpage}).  
This gives a methodical way of checking linear dependence or independence of winding elements.

\bibliographystyle{crelle}
\bibliography{Eisensteinquestion.bib}

\end{document}